\documentclass[11pt]{amsart}
\usepackage[T1]{fontenc}

\usepackage{SSdefn, SSams}
\usepackage{eucal, mathtools}
\usepackage{tikz}
\usepackage{tikz-cd}

\newcommand{\FI}{\mathrm{FI}}
\newcommand{\ab}{\mathrm{ab}}
\newcommand{\aF}{\mathbb{F}}

\newcommand{\Irr}{\mathrm{Irr}}

\newcommand{\VI}{\mathrm{VI}}
\newcommand{\ModVI}{\Mod_{\VI}}

\newcommand{\newshift}{\bar{\Sigma}}

\newcommand{\sat}{\mathrm{sat}}

\newcommand{\adm}{\mathrm{adm}}

\newcommand{\he}{\hat{e}}

\newcommand{\gen}{\mathrm{gen}}

\newcommand{\rHom}{\rR\!\Hom}

\newcommand{\bmu}{\boldsymbol{\mu}}

\newcommand{\V}[1]{\pi^{-1}(#1)}

\DeclareMathOperator{\Ext}{Ext}
\DeclarePairedDelimiter\abs{\lvert}{\rvert}

\newcommand{\coloneq}{\mathrel{\mathop:}\mkern-1.2mu=}

\title{$\VI$-modules in non-describing characteristic, part II}

\author{Rohit Nagpal}
\address{Department of Mathematics, University of Chicago, Chicago, IL}
\email{\href{mailto:nagpal@math.uchicago.edu}{nagpal@math.uchicago.edu}}
\urladdr{\url{http://math.uchicago.edu/~nagpal/}}


\begin{document}


\begin{abstract}
We classify all irreducible generic $\VI$-modules in non-describing characteristic. Our result degenerates to yield a  classification of irreducible generic $\FI$-modules in arbitrary characteristic. Our result can also be viewed as a classification theorem for a natural class of representations of $\GL_{\infty}(\aF_q)$. 
\end{abstract}

\maketitle

\tableofcontents

\section{Introduction}

\subsubsection*{Notation} Set $\aF = \aF_q$, and let $\GL_n$ be the $n$th general linear group over $\aF$. Let $\bk$ be a field in which $q$ is invertible. 


A $\VI$-module $M$ is a functor $M \colon \VI \to \Mod_{\bk}$, where $\VI$ is the category of finite dimensional $\aF$-vector spaces with injective linear maps. Since $\GL_n$ acts on $M(\aF^n)$, we can view $M$ as a sequence whose $n$th member is a $\bk[\GL_n]$-module.  Let $\cI$ denote the left adjoint to the natural restriction \[\Mod_{\VI} \to \prod_{n \ge 0}\Mod_{\bk[\GL_n]}.\] We call $\VI$-modules of the form $\cI(\Theta)$ {\bf induced}, and we refer to $\VI$-modules admitting a finite filtration with induced graded pieces as {\bf semi-induced}.  The category $\Mod_{\VI}$ naturally contains a localizing subcategory $\Mod_{\VI}^{\tors}$ whose members are called {\bf torsion} $\VI$-modules.  We denote the Serre quotient category \[ \Mod_{\VI}/\Mod_{\VI}^{\tors} \] by $\Mod_{\VI}^{\gen}$ and refer to its objects as {\bf generic $\VI$-modules}. See \cite{VI1} for more on these definitions.


 In non-describing characteristic, both the categories $\Mod_{\VI}^{\tors}$ and $\Mod_{\VI}^{\gen}$ are of Krull dimension 0. Isomorphism classes $\Irr(\Mod_{\VI}^{\tors})$ of irreducible objects in $\Mod_{\VI}^{\tors}$ are easy to understand and are in a natural one-to-one correspondence with \[ \bigsqcup_{n \ge 0} \Irr(\Mod_{\bk[\GL_n]}).  \] Surprisingly, the two categories $\Mod_{\VI}^{\tors}$ and $\Mod_{\VI}^{\gen}$ are equivalent if $\bk$ is a field of characteristic 0;  see \cite[Theorem~3.6]{GLX}.  But this equivalence breaks down if $\bk$ is not of characteristic 0, the reason being --  $\Mod_{\VI}^{\tors}$ has finitely generated injective objects but $\Mod_{\VI}^{\gen}$ doesn't. We construct all irreducibles in $\Mod_{\VI}^{\gen}$ in non-describing characteristic. 

\begin{theorem}
	\label{intro:thm:irreducible-correspondence}
	\label{intro:thm:equivalent-characterizations-irreducible}
	Suppose $q$ is invertible in $\bk$. Then there is a natural one-to-one correspondence  \[ \rL \colon \bigsqcup_{n \ge 0} \Irr(\Mod_{\bk[\GL_n]}) \to \Irr(\Mod_{\VI}^{\gen}). \] Let $\Theta$ be an irreducible representation of $\GL_n$. The following are equivalent descriptions of $\rL(\Theta)$:
	\begin{enumerate}[\rm \indent (a)]
		\item $\rL(\Theta)$ is the socle of $ \cI( \Theta)$ in the category of generic $\VI$-modules. 
		
		\item $\rL(\Theta)$ is the generic $\VI$-module given by the kernel of the intersections of all maps from $ \cI( \Theta)$ to $\VI$-modules generated in degrees $<n$.
	\end{enumerate}
\end{theorem}

%
%
%
%
%

When $\bk$ is an algebraically closed field of characteristic 0, this description of $\rL(\Theta)$ becomes particularly simply, and is known to experts (and also follows from the equivalence in  \cite[Theorem~3.6]{GLX}). We now provide this description for completeness. Recall that the isomorphism classes of irreducible representations of $\GL_n$ are parametrized by partition valued functions --  Let $\cC_n$ be the isomorphism classes of cuspidal representations of $\GL_n$, and set $\cC = \sqcup_{n \ge 1} \cC_n$. If $\rho \in \cC_n$, we set $|\rho| = n$. Let $\cP$ be the set of partitions. Given a partition  $\lambda$, we set $|\lambda| = n$ if $\lambda$ is a partition of $n$. Given a function $\bmu \colon \cC \to \cP$, we set $|\bmu| = \sum_{x \in \cC} |x| |\bmu(x)|$.  The isomorphism classes of irreducible representations of $\GL_n$ are in bijection with the set of functions $\bmu$ satisfying $|\bmu| = n$. We fix an irreducible representation $\Theta_{\bmu}$ corresponding to each partition function $\bmu$. Let $\iota \in \cC_1$ be the trivial representation of $\GL_1$. For a partition function $\bmu$ with $\bmu(\iota) = \lambda$, we define another partition function $\bmu[n]$ by \[ \bmu[n](\rho) = \begin{cases}
(n - |\bmu|, \lambda_1, \lambda_2, \ldots) &\mbox{if } \rho = \iota \\
\bmu(\rho) &\mbox{if } \rho \neq \iota.
\end{cases} \] This definition makes sense only if $n \ge |\bmu| + \lambda_1$. The following result can be easily obtained using the results in  \cite{gan-watterlond-VI}.

\begin{theorem}[\cite{gan-watterlond-VI}]
	Suppose $\bk$ is a field of characteristic $0$. Let $\bmu$ be a partition valued function. Suppose $|\bmu| = d$. Suppose $\bmu(\iota) = \lambda$. Let $\cL(\Theta_{\bmu})$ denote the $\VI$-submodule of $\cI(\Theta_{\bmu})$ given by \[ \cL(\Theta_{\bmu})(\aF^n)  = \begin{cases} \Theta_{\bmu[n]}
	&\mbox{if } n \ge d + \lambda_1 \\
	0 &\mbox{otherwise.}
	\end{cases} \] Then the unique irreducible generic $\VI$-module $\rL(\Theta_{\bmu})$ is the image of $\cL(\Theta_{\bmu})$ in the generic category.
\end{theorem}

Our results extend the theorem above to non-describing characteristic and also strengthen it in characteristic $0$. To be more precise, let $\rT \colon \Mod_{\VI} \to \Mod_{\VI}^{\gen}$ denote the {\bf localization functor}, and let $\rS \colon \Mod_{\VI}^{\gen} \to  \Mod_{\VI}$ be its right adjoint ({\bf the section functor}). We have the following result.

\begin{theorem}
	\label{intro:thm:saturation}
	Suppose $q$ is invertible in $\bk$.  Let $\Theta$ be an irreducible representation of $\GL_d$. Let $\cL(\Theta)$ denote the $\VI$-module given by the kernel of the intersections of all maps from $ \cI( \Theta)$ to semi-induced $\VI$-modules generated in degrees $<d$. Then we have the following: \begin{enumerate}[\rm \indent (a)]
		\item $\rS(\rL(\Theta)) = \cL(\Theta)$. In other words, the image of a nonzero  $\VI$-submodule $M \subset \cI(\Theta)$  is ismorphic to $\rL(\Theta)$ in the generic category if and only if $M \subset \cL(\Theta)$.
		\item $\cL(\Theta)$ is generated in degrees $\le 2d$. In fact, Castelnuovo--Mumford regularity of  $\cL(\Theta)$ is at most $2d$. 
		
		\item There is a polynomial $P$ of degree exactly $d$ such that \[ \dim_{\bk} \cL(\Theta)(\aF^n) = P(q^n)  \text{ for } n > 2d - 2.\] 
	\end{enumerate} Moreover, if $\bk$ is a field of characteristic $0$ and $\Theta = \Theta_{\bmu}$, then the two descriptions of  $\cL(\Theta_{\bmu})$ as in this theorem and the previous theorem agree.
\end{theorem}

\begin{remark}
	We note that the category of $\VI$-modules is locally noetherian; see  \cite{putman-sam} or \cite{catgb}. It follows that  $\cL(\Theta)$ is a finitely generated $\VI$-module.
\end{remark}

\begin{question}
\label{intro:question:saturation}	
	Suppose $q$ is invertible in $\bk$, and let $\cL(\Theta)$ be as in the theorem above. \begin{enumerate}[\rm \indent (a)]
		\item What is the precise degree of generation of $\cL(\Theta)$ as a function of $\Theta$? In characteristic 0, it is easy to see that the answer to this question is $|\bmu| + \lambda_1$. 
		
		\item We provide an explicit generator for the $\GL_n$ representation $\cL(\Theta)(\aF^n)$ for any $n \ge d + (1 + q + \cdots + q^{d-1})$. But we prove that $\cL(\Theta)$ is generated in degrees $\le d$. Can we construct an explicit generator for $\cL(\Theta)(\aF^{2d})$?
		 
		\item Is it true that, for any irreducible representation $\Theta$,  the $\GL_n$-representation $\cL(\Theta)(\aF^n)$ is  irreducible  for infinitely many $n$.
		\item Can we calculate the dimension of $\cL(\Theta)(\aF^n)$ as a function of $\Theta$?
		
		\item Can we calculate extensions $\Ext(\rL(\Theta_1), \rL(\Theta_2))$ explicitly? Or the local cohomology groups for $\rL(\Theta)$ explicitly?
	\end{enumerate}
\end{question}

\begin{remark}
	In characteristic 0,  answers to all the parts, except Part (e), of the question above can easily be spelled out. It is possible to follow arguments as in \cite{symc1} to answer part (e) to some extent but we do not pursue it in this paper. 
\end{remark}

\subsection{The main idea and a partial result in defining characteristic }  Let $(P, \le)$ be a poset. A subset $S \subset P$  is {\bf cofinal} if  for any $x \in P$ there is a $y \in S$  such that $x \le y$. By the {\bf submodule-lattice} of a module $M$, we mean the poset of submodules of $M$ under reverse inclusion. Let $\Theta$ be an irreducible representation of $\GL_d$. We provide an explicit construction of a sequence of nonzero submodules  \[M_{d, \Theta} \supset M_{d+1, \Theta} \supset M_{d+2, \Theta} \supset \ldots  \] of $\cI(\Theta)$ which, together with the $0$ submodule, form a cofinal subset in the submodule lattice of the induced module $\cI(\Theta)$.   The following result does not need the non-describing characteristic assumption. 

\begin{theorem}
	\label{intro:thm:cofinal}
	Let $\bk$ be a field of arbitrary characteristic. Let $\Theta$ be an irreducible representation of $\GL_d$ over $\bk$. Suppose $M \subset \cI(\Theta)$ is any nonzero submodule.  Then $M_{n, \Theta} \subset M$ for $n$ large enough. In particular, if $M(\aF^n) \neq 0$, then $M_{m, \Theta} \subset M$ for all $m \ge n$. 
\end{theorem}

We then use the non-describing characteristic assumption to show that this cofinal sequence stabilizes up to torsion, that is, it stabilizes in the Serre quotient category $\ModVI^{\gen}$. 

\begin{theorem}
	\label{intro:thm:collapse}
	Suppose $q$ is invertible in $\bk$. Let $\Theta$ be any representation of $\GL_d$. Then the descending chain $M_{d, \Theta} \supset M_{d+1, \Theta} \supset \ldots$ stabilizes in $\ModVI^{\gen}$. In fact, $M_{d,  \Theta}/M_{n,  \Theta}$ is supported in degrees $< n + q^{d(n-d)}(1 + q + \cdots + q^{d-1})$. In other words,  $M_{n, \Theta} = M_{d,  \Theta}$ in $\ModVI^{\gen}$ for each $n \ge d$. 
\end{theorem}

The two theorem above let us conclude that, in non-describing characteristic,  the image of $M_{n, \Theta}$ for any $n \ge d$ in the generic category is the irreducible $\rL(\Theta)$ from Theorem~\ref{intro:thm:equivalent-characterizations-irreducible}.  The claim that these form a complete set of irreducibles of the generic category then follows quite formally from the structure theory for $\VI$-modules developed in the first paper of this sequel. The novelty of this paper lies in the two theorems above whose proofs occupy \S \ref{sec:cofinal} and \S \ref{sec:stabilization} respectively.

We note that the theory of $\VI$-modules in defining characteristic is much harder and only very little of the structure theory is known. Our Theorem~\ref{intro:thm:cofinal} provides a hint in this direction. We pose some conjectures in equal characteristic. 

\begin{question}
	Suppose $\bk = \aF$. Let $\Theta$ be an irreducible representation of $\GL_d$. \begin{enumerate}[\rm \indent (a)]
		\item Is it true that \[\lim_{m \to \infty} \frac{1}{q^{dm}} \dim_{\bk} \left( \frac{M_{n, \Theta}}{M_{n+1, \Theta}} \right) (\aF^m) = 0?\]
		
		\item When $d = 1$, is it true that $\frac{M_{d, \Theta}}{M_{n, \Theta}}$ is a polynomial growth functor? If so, what is the degree?
		
		\item Can we classify all finitely generated polynomial growth functors? (It is proven in \cite{VI1} that any polynomial growth functor in non-describing characteristic is eventually constant. )
	\end{enumerate}
\end{question}

\subsection{The case of $\FI$-modules} For a partition $\mu$, define another partition $\mu[n]$ by \[ \mu[n] = 
(n - |\mu|, \mu_1, \mu_2, \ldots). \]  In characteristic 0,  Church--Ellenberg--Farb \cite{fimodules} showed that for every finitely generated $\FI$-module $M$, there is a finite set $F$ of partitions such that \[M_n  = \bigoplus_{\mu \in F} \bM_{\mu[n]}  \] for large enough $n$, where $\bM_{\lambda}$ denote the Specht module corresponding to the partition $\lambda$. Sam--Snowden \cite{symc1} showed a stronger result that \[ \cL(M_{\mu}) \coloneq \bigoplus_{n \ge |\mu| + \mu_1} M_{\mu[n]}  \] is an irreducible in the category of generic $\FI$-modules, and that all irreducibles in this category are of this form. This establishes a natural one-to-one correspondence between $\Irr(\Mod_{\FI}^{\gen})$ and $\bigsqcup_{n \ge 0} \Irr(\Mod_{\bk[S_n]})$. Church--Ellenberg--Farb's result holds in positive characteristic if we pass to the Grothendieck group and allow negative coefficients; see \cite{virtual-stab}. But away from characteristic 0, a classification of irreducibles for the category of generic $\FI$-modules was not known previously. Our method for classification of irreducibles for $\VI$-modules degenerates to yield the following result.

\begin{theorem}
	Suppose $\bk$ is an arbitrary field. Then there is a natural one-to-one correspondence between $\Irr(\Mod_{\FI}^{\gen})$ and $\bigsqcup_{n \ge 0} \Irr(\Mod_{\bk[S_n]})$. 
	
\end{theorem}

Our explicit description of the irreducibles in the generic category and  Theorem~\ref{intro:thm:equivalent-characterizations-irreducible} and Theorem~\ref{intro:thm:saturation} also degenerate to yield analogous results for $\FI$-modules in arbitrary characteristic. 

\begin{remark}
We do not provide separate proofs in the case of $\FI$-modules as they can easily be obtained by setting $q=1$ in our proofs for $\VI$-modules. We also note that   Theorem~\ref{intro:thm:collapse} is trivial in the case of $\FI$-modules but is one of the main technical results in this paper. Moreover, our argument for Theorem~\ref{intro:thm:cofinal} can be thought of as a $\GL$ version of some of the combinatorial results in \cite{castelnuovo-regularity}. 
	
	The analogue of Question~\ref{intro:question:saturation}, away from characteristic 0, is completely open for $\FI$-modules as well, except for Part (b) which has no content as the two functions become equal when we plug in $q = 1$. However, in characteristic 0, answer to this question and all the results in this paper are known for $\FI$-module; see \cite{symc1}.
\end{remark}

\subsection{Relations to $\GL_{\infty}$ representations and Deligne categories}
\label{sec:infinity-perspective}

The natural inclusion $\aF^n \to \aF^{n+1}$ of vector spaces induces a natural inclusion $\GL_n \to \GL_{n+1}$ of groups. By $\aF^{\infty}$ and $\GL_{\infty}$, we denote the direct limits  given by these inclusions. Let $P_n$ denote the subgroup of $\GL_{\infty}$ consisting of elements that fix $\aF^n \subset \aF^{\infty}$ pointwise. We call a $\bk[\GL_{\infty}]$-module $M$ {\bf admissible} if for each $x \in M$ there exists an $n$ such that every $\sigma \in  P_n$ fixes $x$. There is a natural equivalence of categories between $\Mod_{\VI}^{\gen}$ and the category  $\Mod_{\bk[\GL_{\infty}]}^{\adm}$ of admissible $\bk[\GL_{\infty}]$-modules. To see this, note that we have two functors \begin{align*}
\Psi \colon \Mod_{\VI} \to \Mod_{\bk[\GL_{\infty}]}^{\adm} \\
\Phi  \colon \Mod_{\bk[\GL_{\infty}]}^{\adm} \to \Mod_{\VI}
\end{align*} given by \begin{align*}
&\Psi(M) = \lim_n M(\aF^n) \\
&\Phi(M)(\aF^n) = M^{P_n}.
\end{align*} It is an easy verification that $\Psi$ factors through $\Mod_{\VI}^{\gen}$ and induces an equivalence $\Psi' \colon \Mod_{\VI}^{\gen} \to \Mod_{\bk[\GL_{\infty}]}^{\adm}$ where the inverse is obtained by composing $\Phi$ with the localization functor $\rT$. Thus the following result is a corollary of Theorem~\ref{intro:thm:irreducible-correspondence} and Theorem~\ref{intro:thm:saturation}: 

\begin{theorem}
	Suppose $q$ is invertible in $\bk$. Then we have a one-to-one correspondence \[\bigsqcup_{n \ge 0} \Irr(\Mod_{\bk[\GL_n]}) \to \Mod_{\bk[\GL_{\infty}]}^{\adm} \] given by \[\Theta \mapsto \lim_n \cL(\Theta)(\aF^n).\] Moreover, we have $\Phi \circ\Psi'(\rL(\Theta))  = \cL(\Theta)$.
\end{theorem}


\begin{remark}
	We do not talk about $\GL_{\infty}$-perspective for the rest of the paper as our results follow from the corresponding results on $\VI$-modules and the equivalence of category mentioned above. Our method also yield an analogues result for the infinite symmetric group.
\end{remark}

Deligne \cite{del}  constructed families of rigid symmetric Karoubian tensor  categories $\ul{\Rep}(S_t)$ for $t \in \bC$. These categories interpolate $\Rep(S_n)$ and can be thought of as the representation theory of the symmetric group in complex dimension $t$. Deligne and Milne \cite{deligne-milne} constructed the $\GL_{\infty}$ version  $\ul{\Rep}(\GL_t)$ of these interpolation categories. These categories are not abelian when $t$ is an integer. Comes and Ostrik \cite{comes-ostrik} constructed the abelian envelope $\ul{\Rep}^\ab(S_t)$ of $\ul{\Rep}(S_t)$.  Karoubian and abelian versions of Deligne categories have been constructed for several other sequences of groups, for example, $\GL(m|n)$ \cite{entova}. Deligne's construction is not very well-behaved away from characteristic 0 in the sense that it does not capture enough of the modular representation theory of the symmetric groups. Harman \cite{nate} constructed  $\ul{\Rep}_{\bk}(S_t)$ for $t \in \bZ_p$ over a field $\bk$ of characteristic $p$, resolving a conjecture of Deligne from one of his letter to Ostrik. This category captures more refined modular representation theory of symmetric groups and can be thought of as ``modular representation theory of symmetric groups in $p$-adic dimension''. At least in characteristic 0, a relation between $\Rep(S_{\infty})$ and $\ul{\Rep}^{\ab}(S_t)$ was provided in \cite{barter} where it was proven that there is an exact symmtric monoidal faithful functor \[\Rep^{\adm}(S_{\infty}) \to  \ul{\Rep}^{\ab}(S_t).\]

\begin{question}
	Can we construct a Deligne category for modular representation theory of general linear group in $p$-adic rank, at least in non-describing characteristic? If so, do we have an exact symmetric monoidal functor from $\Mod_{\bk[\GL_{\infty}]}^{\adm}$ to this Deligne category?
\end{question}


\subsection*{Acknowledgements}
We thank Inna Entova-Aizenbud and Steven V Sam for discussions on the $\GL_{\infty}$-perspective which is outlined in \S \ref{sec:infinity-perspective}.

\section{A cofinal sequence in the submodule-lattice of an induced module}
\label{sec:cofinal}


 Let $\Theta$ be an irreducible representation of $\GL_d$. Our aim in this subsection is to show that there is a natural sequence of nonzero submodules of $\cI(\Theta)$ \[M_{d, \Theta} \supset M_{d+1, \Theta} \supset M_{d+2, \Theta} \supset \ldots  \] which, together with the $0$ submodule, form a cofinal subset in the submodule lattice of the induced module $\cI(\Theta)$.   This result does not need non-describing characteristic assumption. In the next section, we shall use non-describing characteristic assumption to show that this sequence stabilizes up to torsion, that is, it stabilizes in the Serre quotient category $\ModVI^{\gen}$. 
 
 We recall that $\cI(\Theta) = \bk[\Hom_{\VI}(\aF^d, - )]\otimes_{\GL_d} \Theta$. In particular, any element of $\cI(\Theta)(X)$ can be written as a $\bk$-linear combination of elements of the form $[f] \otimes \theta$ where $f \colon \aF^d \to X$ is an $\aF$-linear injection and $\theta \in \Theta$. When $\Theta$ is the regular representation of $\GL_d$ then we denote $\cI(\Theta)$ by simply $\cI(d)$. In other words, we have $\cI(d) = \bk[\Hom_{\VI}(\aF^d, - )]$. To be able to define the submodule $M_{n, \Theta}$ we need some preliminaries:

Fix a morphism $f \colon \aF^d \to X$ in $\VI$. In other words, $f$ is an $\aF$-linear injective map from $\aF^d$ to a vector space $X$. Fix an ordered basis $(\cB_f, \prec)$ of the image $\im(f)$ of $f$. Denote by $\cV_f$ the set of triples $(\alpha, C, C')$ satisfying the following \begin{enumerate}[\rm \indent (a)]
	\item $\alpha  \in \cB_f$,
	\item $C$ is a complement of the line $\aF \alpha$ in $\im(f)$ satisfying $(\cB_f)_{\prec \alpha} \subset C$, and 
	\item $C'$ is a complement of $W$ in $V$. 
\end{enumerate} For $v \in \cV_f$, we denote the corresponding triple by $(\alpha_v, C_v, C'_v)$. Let $\cE_f \coloneq \sum_{v \in \cV_f} \aF e_{v}$ be the vector space freely generated by $\cV_f$.


For $v \in \cV_f$, let $\sigma_v \in \GL(X + \cE_f)$ be the element that takes $\alpha_v$ to $e_v$ and fixes the following pointwise: $\alpha_w - e_w$ for each $w \in \cV_f$, $C_v$ and $C'_v$. Let $<$ be a linear order on $\cV_f$ (which, for now, is independent of the order on $\cB_f$). For a subset $S$ of $\cV_f$, let $L_S^<$ be the descending product $\prod_{v \in S} (\id - \sigma_v)$. This implies that if $S'$ is an initial segment of $S$ then $L_S^< = L_{S \setminus S'}^< L_{S'}^<$. We will suppress the superscript `$<$' when the order is implicit. Note here that, given a $\VI$-module $M$ over $\bk$, $L_S$ induces a $\bk$-endomorphism on $M(X + \cE_f)$ which is functorial in $M$. 



Given a morphism $g \colon Y \to X$ and an object $E$ in $\VI$, we denote the direct sum of $g$ and the unique morphism $0 \to E$ by $g^{E}$. Thus $g^E$ is a morphism from $Y$ to $X + E$ whose image is contained in $X$. Let $g \colon Y \to X$ be a $\VI$-morphism, and let $w \in \cV_f$. By $\langle g, \alpha_w \rangle_w \colon Y \to \aF$ we denote the unique $\aF$-linear map such that for each $y \in Y$, $g(y)$ can be written as $\langle g(y), \alpha_w \rangle_w \alpha_w + c$ for some $c \in C_w + C'_w$.

\begin{lemma}
	\label{lem:inner-sum}
	Let $g \colon Y \to X$ be a $\VI$-morphism. Let $S$ be a subset of $\cV_f$. Then we have \begin{displaymath}
	\Big(\prod_{w \in S}\sigma_{w}\Big) g^{\cE_f} = g^{\cE_f} + \sum_{w \in S}\langle g, \alpha_{w} \rangle_{w}(e_{w} - \alpha_{w})
	\end{displaymath} where the product is descending.
\end{lemma} 
\begin{proof} We prove the result by induction on the size of $S$. Let $y \in Y$. Then $g^{\cE_f}(y) = g(y) = \langle g(y), \alpha_{w} \rangle_{w} \alpha_{w} + c$ where $c \in C_{w} + C'_{w}$. This implies that \[\sigma_{w} g^{\cE_f}(y)   =  \langle g(y), \alpha_{w} \rangle_{w} e_{w} + c = \langle g(y), \alpha_{w} \rangle_{w} (e_{w} - \alpha_{w}) + g^{\cE_f}(y).\] Thus we have $\sigma_{w} g^{\cE_f}    = \langle g, \alpha_{w} \rangle_{w} (e_{w} - \alpha_{w}) + g^{\cE_f}$ proving the result for $n = 1$. The general result follows from it because $\sigma_{w'}$ fixes $(e_{w} - \alpha_{w})$ for each $w' \in S$.
\end{proof}

\begin{lemma}
	\label{lem:order-independence}
	Let $g \colon Y \to X$ be a morphisms in $\VI$. Let $<_1, <_2$ be two linear orders on $\cV_f$ and $S$ be a subset of $\cV_f$. Then $L_S^{<_1}([g^{\cE_f}]) = L_S^{<_2}([g^{\cE_f}])$.
\end{lemma}
\begin{proof} This is immediate from Lemma~\ref{lem:inner-sum}. 
\end{proof}
%
%


We now define $M_{n, \Theta}$ for an arbitrary representation $\Theta$ of $\GL_d$. Let $f_n \colon \aF^d \to \aF^n$ be the natural inclusion -- the map that takes the standard basis of $\aF^d$ to an initial segment of the standard basis of $\aF^n$. We assume that $(\cB_{F_n}, \prec)$ is this initial segment, and define $\cV_{f_n}, \cE_{f_n}, L_{\cV_{f_n}}$ as in the previous subsection. We define $M_{n, \Theta}$ to be the $\VI$-submodule of $\cI(\Theta)$ generated by elements of the form $L_{\cV_{f_n}}([f_n^{\cE_{f_n}}] \otimes \theta)$ where $\theta \in \Theta$. Note here that that $L_{\cV_{f_n}}([f_n^{\cE_{f_n}}] \otimes \theta) = L_{\cV_{f_n}}([f_n^{\cE_{f_n}}])\otimes \theta$.


\begin{proposition}
	\label{prop:descending}
	We have $M_{n, \Theta} \supset M_{n+1, \Theta}$.
\end{proposition}

The proposition above follows immediate from this more general lemma.

\begin{lemma} 
	\label{lem:descending-modules}
	Suppose $\dim_{\aF} Y \ge \dim_{\aF} X$. Let $f \colon \aF^d \to X$, $g \colon \aF^d \to Y$ be morphisms in $\VI$. If there exists a $\VI$-morphism $h \colon X \to Y$ that takes $\cB_f$ to $\cB_g$ in an order preserving way and satisfies $g = h \circ f$, then $L_{\cV_f}([f^{\cE_f}])$ generate $L_{\cV_{g}}([g^{\cE_g}])$ in $\cI(d)$.
\end{lemma}
\begin{proof} Let $K$ be a complement of $\im(h)$ in $Y$.  Let $\phi \colon \cV_f \to \cV_g$ be the injective map taking $(\alpha, C, C')$ to $(h(\alpha), h(C), h(C') + K)$. Fix a linear order $<_1$ on $\cV_f$ and let $<_2$ be a linear order on $\cV_g$ such that $\phi$ is order preserving and $\im(\phi)$ is an initial segment. Let $\wt{\phi} \colon \cE_f \to \cE_g$ be the natural injection induced by $\phi$, that is, we have $\wt{\phi}(e_v) = e_{\phi(v)}$. Denote the direct sum map $(h + \wt{\phi}) \colon (X + \cE_f) \to (Y + \cE_g)$ by $\wt{h}$. We claim that $\wt{h} \circ \sigma_v = \sigma_{\phi(v)}\circ \wt{h}$ as a linear map from $(X + \cE_f) \to (Y + \cE_g)$. It suffices to check that $\wt{h} \circ \sigma_v(t) = \sigma_{\phi(v)} \circ \wt{h}(t)$ for $t$ in a spanning set for $X + \cE_f = \aF \alpha_v + C_v +C'_v + \cE_f$. First suppose $t = \alpha_v$. Then we have \[\wt{h} \circ \sigma_v(\alpha_v) = \wt{h}(e_v) = \wt{\phi}(e_v) = e_{\phi(v)} = \sigma_{\phi(v)}(\alpha_{\phi(v)}) =  \sigma_{\phi(v)}(h(\alpha_v)) = \sigma_{\phi(v)} \circ \wt{h}(\alpha_v).   \]  Now suppose $t \in C_v + C'_v$. Then we have \[ \wt{h} \circ \sigma_v(t) = \wt{h}(t) = h(t) = \sigma_{\phi(v)}(h(t)) = \sigma_{\phi(v)} \circ \wt{h}(t). \] Finally, suppose $t = \alpha_w - e_w$ for some $w \in \cE_f$. Then we have \[ \wt{h} \circ \sigma_v(\alpha_w - e_w) = \wt{h}(\alpha_w - e_w) = h(\alpha_w) - \wt{\phi}(e_w) = \alpha_{\phi(w)} - e_{\phi(w)} =   \sigma_{\phi(v)}(\alpha_{\phi(w)} - e_{\phi(w)}) = \sigma_{\phi(v)} \circ \wt{h}(\alpha_w - e_w). \] We have proven that $\wt{h} \circ \sigma_v(t) = \sigma_{\phi(v)} \circ \wt{h}(t)$ holds for $t$ in a spanning set for $X + \cE_f$, and so the claim holds.

	Now note that $\wt{h} \circ f^{\cE_f} = (h \circ f)^{\cE_g} = g^{\cE_g}$. Thus the claim in the paragraph above implies that $\wt{h}_{\star}(L_{\cV_f}^{<_1}([f^{\cE_f}])) = L_{\im(\phi)}^{<_2}([g^{\cE_g}])$. This shows that $L_{\cV_f}^{<_1}([f^{\cE_f}])$ generates $L_{\im(\phi)}^{<_2}([g^{\cE_g}])$. Since $\im(\phi)$ is an initial segment of $\cV_g$, we have $L_{\cV_g}^{<_2}([g^{\cE_g}]) = L_{\cV_g \setminus \im(\phi)}^{<_2}L_{\im(\phi)}^{<_2}([g^{\cE_g}])$ which shows that $L_{\im(\phi)}^{<_2}$ generates  $L_{\cV_g}^{<_2}([g^{\cE_g}])$. Thus $L_{\cV_f}^{<_1}([f^{\cE_f}])$ generates $L_{\cV_g}^{<_2}([g^{\cE_g}])$. This is independent of the orders $<_1, <_2$ by Lemma~\ref{lem:order-independence}, completing the proof.
\end{proof}

%

We now recall and prove the following theorem from the introduction (copy of Theorem~\ref{intro:thm:cofinal}).

\begin{theorem}
	\label{thm:cofinal}
	 Suppose $\bk$ is an arbitrary field. Let $\Theta$ be an irreducible representation of $\GL_d$. Suppose $M \subset \cI(\Theta)$ is any nonzero submodule.  Then $M_{n, \Theta} \subset M$ for $n$ large enough. In particular, if $M(\aF^n) \neq 0$, then $M_{m, \Theta} \subset M$ for all $m \ge n$.
\end{theorem}



We need a few lemmas.

\begin{lemma} 
	\label{lem:distinct-images}
	Let $f, g \colon \aF^d \to X$ be $\VI$-morphisms. Suppose $\theta \in \Theta$ is nonzero. Then $L_{\cV_f} ([g^{\cE_f}] \otimes \theta) = 0$ in $\cI(\Theta)$ if and only if $\im(g) \neq \im(f)$. In particular, $M_{n, \Theta} \neq 0$ if $n \ge d$.
\end{lemma}
\begin{proof}	
	First suppose $W' \coloneq \im(g) \neq W \coloneq \im(f)$. Since $W' \neq W$ and $\dim_{\aF}{W} = \dim_{\aF}{W'}$, there exists a $v \in \cV_f$ such that $\alpha_v \notin W'$ and $W' \subset C_v + C'_v$. It suffices to show that $(\id - \sigma_v)$ kills the element $(\prod_{w \in S} \sigma_w)[g^{\cE_f}]$ for any finite subset $S$ of the initial segment $(\cV_f)_{<v}$. By Lemma~\ref{lem:inner-sum}, we have $(\prod_{w \in S} \sigma_w) g^{\cE_f} =  g^{\cE_f} + \sum_{w \in S} \langle g, \alpha_w \rangle_{w}(e_w - \alpha_w ) $. By definition, $\sigma_v$ fixes each $(e_w - \alpha_w )$. Moreover, since the image of $g^{\cE_f}$ is contained in $C_v \oplus C'_v \subset V \oplus \cE_f$, we see that $\sigma_v$ fixes $g^{\cE_f}$ as well. Thus $\sigma_v$ fixes $(\prod_{w \in S} \sigma_w)g^{\cE_f}$, as desired.
	
	Next suppose $\im(g) = \im(f) = W$. For a subset $S$ of $\cV_f$, set $g_S \coloneq  g^{\cE_f} + \sum_{w \in S}\langle g, \alpha_{w} \rangle_{w}(e_{w} - \alpha_{w})$. By Lemma~\ref{lem:inner-sum}, it suffices to show that if $S \neq S'$ then $g_S$ and $ g_{S'}$ have distinct images. To see this, let $v \in S \setminus S'$ and let $C$ be the direct sum of $X$ and $\sum_{w \neq v} \aF e_w$. Then the image of $g_{S'}$ lie in $C$ but that of $g_S$ doesn't lie in $C$ (the functional $\langle g, \alpha_{v} \rangle_{v}$ is nonzero because $\im(g) = \im(f)$). This completes the proof.
\end{proof}

\begin{proof}[Proof of Theorem~\ref{thm:cofinal}]
	Since $M \neq 0$ there exists an $n$ such that $M(\aF^n) \neq 0$, and so it suffices to prove the second assertion. By Proposition~\ref{prop:descending}, it suffices to show that $M_{n, \Theta} \subset M$. In other words, it suffices to show that if $\theta \in \Theta$ is arbitrary, then $L_{\cV_{f_n}}([f_n^{\cE_{f_n}}] \otimes \theta) \in M$.
	
	 Let $x \in M(\aF^n)$ be a nonzero element. For every subspace $W$ of $\aF^n$, choose a $\VI$-morphism $f_W \colon \aF^d \to \aF^n$. We assume that the natural inclusion $f_n \colon \aF^d \to \aF^n$ is among these choices of morphisms. Now every element of $\cI(\Theta)(\aF^n)$ can be written uniquely as $\sum_{W} [f_W] \otimes \theta_W$ where $\theta_W \in \Theta$. In particular, we have $x = \sum_{W} [f_W] \otimes \theta_W$. Since $x \neq 0$ there exists a $W_0$ such that $\theta_{W_0} \neq 0$. 
	
	Now let $\theta$ be an arbitrary element of $\Theta$. Since $\Theta$ is irreducible, we can write $\theta$ as \[\theta  = \sum_{\sigma \in \GL_d} a_{\sigma} \sigma \theta_{W_0}, \] where $a_{\sigma} \in \bk$. For $\sigma \in \GL_d$, let  $\tau_{\sigma} \in \GL_n$ be an automorphism such that $\tau_{\sigma} \circ f_{W_0} = f_n \circ \sigma$. Such an automorphism exists by transitivity of action of $\GL_n$ on $d$-dimensional subspaces of $\aF^n$. It is easy to see that $\im( \tau_{\sigma} f_W) = \im(f_n)$ if and only if $W =  W_0$.  Set $y = \sum_{\sigma \in \GL_d} a_{\sigma} \tau_{\sigma} x$. Clearly, $y \in M(\aF^n)$. We have  \[  \sum_{\sigma \in \GL_d} a_{\sigma} \tau_{\sigma} [f_{W_0}] \otimes \theta_{W_0} = \sum_{\sigma \in \GL_d} a_{\sigma} [f_n \sigma] \otimes \theta_{W_0} = [f_n]  \otimes  \left( \sum_{\sigma \in \GL_d} a_{\sigma} \sigma \theta_{W_0} \right) = [f_n] \otimes \theta.  \]  This shows that $y$ can be written as \[y =  [f_n] \otimes \theta +  \sum_{W \neq \im(f_n) } [f_W] \otimes \theta'_W \]  for some $\theta'_W \in \Theta$. Now let $\ell \colon \aF^n \to \aF^n + \cE_{f_n}$ be the natural inclusion. Since $y \in M(\aF^n)$, we conclude that  $L_{\cV_{f_n}}(\ell_{\star}(y)) \in M(\aF^n + \cE_{f_n})$. We have \[\ell_{\star}(y) = [f_n^{\cE_{f_n}}] \otimes \theta +  \sum_{W \neq \im(f_n) } [f_W^{\cE_{f_n}}] \otimes \theta'_W. \] By Lemma~\ref{lem:distinct-images}, we conclude that \[L_{\cV_{f_n}}(\ell_{\star}(y)) = L_{\cV_{f_n}}([f_n^{\cE_{f_n}}] \otimes \theta) +  \sum_{W \neq \im(f_n) } L_{\cV_{f_n}}([f_W^{\cE_{f_n}}] \otimes \theta'_W) = L_{\cV_{f_n}}([f_n^{\cE_{f_n}}] \otimes \theta) \in M(\aF^n + \cE_{f_n}).  \] This proves that $M_{n, \Theta} \subset M$, finishing the proof.
\end{proof}





\section{Stabilization of the cofinal sequence, modulo torsion, in non-describing characteristic}
\label{sec:stabilization}

In this subsection, we assume that $q= \abs{\aF}$ is invertible in $\bk$. 

Let $f \colon \aF^d \to X$ be a morphism in $\VI$. Let $W$ be the image of $f$ with ordered basis $(\cB_f, \prec)$, and $g \colon \aF^d \to W$ be the restriction of $f$ to $W$. Let $<_{\alpha}$ be a linear order on  complements of the line $\aF \alpha$ in $W$ that contain $(\cB_f)_{\prec \alpha} $, and $<'$ be a linear order on complements of $W$ in $V$. Let $<$ be the lexicographic order on $\cV_f$, that is, to check $v < w$ we first compare $\alpha$ components and then $C$ components and at last the $C'$ components. Order $\cV_g$ using the same lexicographic order as above. There is a natural order preserving (non-strictly) projection map $\pi \colon \cV_f \to \cV_g$ and, for $x \in \cV_g$, we denote the fiber of this map at $x$ by $\V{x}$.More explicitly $\pi$ is given by $(\alpha, C, C') \mapsto (\alpha, C, 0)$.

Let  $\hat{ } \colon X + \cE_f \to X + \cE_f$ be the $\VI$ automorphism that takes $e_v$ to $\he_v \coloneq e_v - \alpha_v$ for each $v$ and fixes $V$ pointwise. Let $\cU_{x}$ be the unipotent subgroup (with respect to the order defined by $<$) of the automorphism group of $\hat{\cE}_{x} \coloneq \aF[\{\he_v \colon v \in \V{x} \}]$, that is, $\sigma \in \cU_{x}$ if and only if for each $v\in \V{x}$ the element $\sigma \he_v -  \he_v$ is a linear combination of $\he_w$ with $v < w \in \V{x}$. In particular, $\cU_{x}$ fixes $\he_{\max(\V{x})}$. Clearly, the size of $\cU_{x}$ is a power of $q$. From now on, we regard $\cU_x$ as a subgroup of $\GL(X +  \cE_f)$ that fixes $X$ and $\he_w$ pointwise for each $w \notin \V{x}$.


\begin{lemma} 
	\label{lem:direct-product-cU}
	If $x \neq y$ then $\cU_{x} \cap \cU_{y}$ and $[\cU_{x}, \cU_{y}]$ are trivial. In particular, $\cU_{x} \cU_{y} = \cU_{x} \times \cU_{y}$.
\end{lemma}
\begin{proof} This is immediate.
\end{proof}

For an initial segment $I$ of $\cV_g$ we define $\cU_{I}$ to be the subgroup $\prod_{x \in I} \cU_{x}$. For $x \in \cV_g$, let $D_x$ be the singleton consisting of $\max(\V{x})$ (where the $\max$ is taken with respect to $<$ on $\cV_f$), and let $D_I$ be the union of $D_x$ for $x \in I$. 

\begin{lemma}
	\label{lem:unipotent-base} 
	Let $x \in \cV_g$ and let $X$ be a subset of $\cV_f$ contained in $\sqcup_{y < x} \V{y}$. Then we have \[\frac{1}{\abs{\cU_{x}}}\sum_{\sigma \in \cU_{x}} \sigma L_{\V{x}}L_X([f^{\cE_f}])=   L_{D_{x}}L_X([f^{\cE_f}]).\]
\end{lemma}
\begin{proof}
	Let $g_{X'} = f^{\cE_f} + \sum_{v \in X' \subset X} \langle f, \alpha_{v}  \rangle_v \he_v$. By Lemma~\ref{lem:inner-sum},  $L_X([f^{\cE_f}])$ is a linear combination of elements of the form $[g_{X'}]$. So let $g = g_{X'}$ for some subset $X'$ of $X$. By linearity, it suffices to show that $\frac{1}{\abs{\cU_{x}}}\sum_{\sigma \in \cU_{x}} \sigma L_{\V{x}}([g])=   L_{D_{x}}([g])$. Note that $\sigma g = g$ for each $\sigma \in \cU_x$. Thus by Lemma~\ref{lem:inner-sum}, we have
	\begin{align*}
	\sum_{\sigma \in \cU_{x}} \sigma L_{\V{x}}([g]) &= \sum_{\sigma \in \cU_{x}} \sigma \left(\sum_{S' \subset \V{x}} (-1)^{\abs{S'}}[g + \sum_{v \in S'} \langle f, \alpha_{v}  \rangle_v \he_v ]  \right)  \\
	&= \sum_{\sigma \in \cU_{x}}  \left(\sum_{S' \subset \V{x}} (-1)^{\abs{S'}}[g + \sum_{v \in S'} \langle f, x_{v}  \rangle_v (\sigma \he_v) ]  \right) \\
	&=  \sum_{S' \subset \V{x}} (-1)^{\abs{S'}} \left(\sum_{\sigma \in \cU_{x}} [g + \sum_{v \in S'} \langle f, \alpha_{v}  \rangle_v (\sigma \he_v) ]  \right) \\
	&= \sum_{S' \subset \V{x}} (-1)^{\abs{S'}} \Phi(S')
	\end{align*} where $\Phi(S')$ is the expression in big round brackets in the fourth expression. Suppose $S' \subset \V{x}$ is nonempty and does not contain $D_{x}$, and let $S'' = S' \sqcup D_{x}$. We claim that $\Phi(S') =  \Phi(S'')$. To see this, first note that the functional $\langle f, \alpha_{v}  \rangle_v$ does not depend on $v \in \V{x}$ because $(\alpha_v, C_v) = (\alpha_w, C_w)$ for $v,w \in \V{x}$.  Now for each $\sigma \in \cU_{x}$ given by $\sigma(\he_v) = e_v + \sum_{w > v} a_{v,w} \he_w$, define $\tau \in \cU_x$ by $\tau(\he_v) = e_v + \sum_{w > v} b_{v,w} \he_w$ where \[ b_{v,w} = \begin{cases}
	a_{v,w} - 1 & \mbox{if $w \in D_x$ and $v = \min{S'}$,}  \\
	a_{v, w} & \mbox{otherwise.}
	\end{cases}  \]  By construction, we have \[ \sum_{v \in S'} \sigma \he_v = \sum_{v \in S''}  \tau \he_v.  \] Since the functional $\langle f, \alpha_{v}  \rangle_v$ does not depend on $v \in \V{x}$, we see that \[ [g + \sum_{v \in S'} \langle f, \alpha_{v}  \rangle_v (\sigma \he_v) ] = [g + \langle f, \alpha_{v}  \rangle_v \sum_{v \in S'}  (\sigma \he_v) ] = [g + \langle f, \alpha_{v}  \rangle_v \sum_{v \in S''}  (\tau \he_v) ] =   [g + \sum_{v \in S''} \langle f, \alpha_{v}  \rangle_v (\tau \he_v) ].    \] It follows immediately that $\Phi(S') =  \Phi(S'')$, establishing the claim. The claim implies that we have \[\sum_{\sigma \in \cU_{x}} \sigma L_{\V{x}}([g])  = \Phi(\emptyset) - \Phi(D_{x}). \] Moreover, we have $\Phi(\emptyset) = \abs{\cU_{x}} [g]$ and $\Phi(D_{x}) = \abs{\cU_{x}} [g + \langle f, \alpha_v \rangle_v (e_v  - \alpha_v)] $ where $v = \max(\V{x})$. This shows that \[\sum_{\sigma \in \cU_{x}} \sigma L_{\V{x}}([g]) = \abs{\cU_{x}} L_{D_{x}}([g]),\] completing the proof.
\end{proof}

\begin{lemma} 
	\label{lem:commutativity}
	Let $\sigma' \in \prod_{y <  x}\cU_{y}$. If $v \in \V{x}$ then $\sigma'$ commutes with $\sigma_v$. 
\end{lemma}
\begin{proof} We have $V + \cE_f = (\sum_{y < x} \hat{\cE}_{y}) + V + (\sum_{x \le y } \hat{\cE}_{y})$. Clearly, $\sigma'$ fixes $V + (\sum_{x \le y } \hat{\cE}_{y})$ pointwise and stabilizes $\sum_{y < x} \hat{\cE}_{y}$. On the other hand, $\sigma_v$ fixes $\sum_{y < x} \hat{\cE}_{y}$ pointwise and stabilizes $V + (\sum_{x \le y } \hat{\cE}_{y})$. The assertion follows from this.
\end{proof}

\begin{lemma}
	\label{lem:unipotent-general} 
	We have $\frac{1}{\abs{\cU_{\cV_g}}}\sum_{\sigma \in \cU_{\cV_g}} \sigma L_{\cV_f}([f^{\cE_f}])=   L_{D_{\cV_g}}([f^{\cE_f}])$.
\end{lemma}
\begin{proof} Let $I$ be an initial segment of $\cV_g$.  We prove by induction on the size of $I$ that \[\sum_{\sigma \in \cU_I} \sigma L_{\V{I}}([f^{\cE_f}]) = \abs{\cU_I} L_{D_I}([f^{\cE_f}]).\] The $\abs{I} = 1$ case follows from Lemma~\ref{lem:unipotent-base}. Next, suppose $I = X \sqcup \{x\}$ for some initial segment $X$ of $\cB_f$. Then we have \begin{align*}
	\sum_{\sigma \in \cU_I} \sigma L_{\V{I}}([f^{\cE_f}]) &= \sum_{(\sigma, \sigma') \in \cU_{x} \times \cU_X} (\sigma, \sigma')  L_{\V{x}}L_{\V{X}}([f^{\cE_f}])\\
	&= \left(\sum_{\sigma \in \cU_{x}} \sigma L_{\V{x}}\right) \left(\sum_{\sigma' \in \cU_X} 
	\sigma' L_{\V{X}} \right) ([f^{\cE_f}]) &\text{by Lemma }\ref{lem:commutativity}\\
	&= \abs{\cU_X} \left(\sum_{\sigma \in \cU_{x}} \sigma L_{\V{x}}\right) L_{D_X} ([f^{\cE_f}])  &\text{by induction }\\
	&= \abs{\cU_X} \abs{\cU_{x} } L_{D_{x}} L_{D_X} ([f^{\cE_f}]) &\text{by Lemma }\ref{lem:unipotent-base}\\
	&= \abs{\cU_I} L_{D_I} ([f^{\cE_f}]) &\text{by Lemma }~\ref{lem:direct-product-cU},
	\end{align*} completing the proof.
\end{proof} 

We now recall and prove the following theorem from the introduction (copy of Theorem~\ref{intro:thm:collapse}).

\begin{theorem}
	\label{thm:collapse}
Let $\Theta$ be any representation of $\GL_d$. Then the descending chain $M_{d, \Theta} \supset M_{d+1, \Theta} \supset \ldots$ stabilizes in $\ModVI^{\gen}$. In fact, $M_{d,  \Theta}/M_{n,  \Theta}$ is supported in degrees $< n + q^{d(n-d)}(1 + q + \cdots + q^{d-1})$. In other words,  $M_{n, \Theta} = M_{d,  \Theta}$ in $\ModVI^{\gen}$ for each $n \ge d$. 
\end{theorem}
\begin{proof} Fix an $n \ge d$. Let $f \colon \aF^d \to \aF^n$ be the natural inclusion. Denote the image of $f$ by $W$, and the restriction $g \colon \aF^d \to W \cong \aF^d$ by $g$.  Let $C'$ be the maximal complement of $W$ with respect to the order $<''$ (as defined in the beginning of this subsection). By definition, $D_{\cV_g} = \cV_g \times \{C'\}$.  Let $\phi \colon \cV_{g} \to D_{\cV_{g}}$ be the bijection taking $(\alpha, C, 0)$ to $(\alpha, C, C')$, and note that $D_{\cV_g} = \im(\phi)$. The map  $\phi$ together with $f$ induces a map $\wt{\phi} \colon \aF^d + \cE_{g} \to \aF^n + \cE_f$. Let $\theta \in \Theta$. Then we have \begin{align*}
\wt{\phi}_{\star}(L_{\cV_{g}}([g^{\cE_{g}}] \otimes \theta)) &= L_{D_{\cV_g}}([f^{\cE_f}] \otimes \theta) \\
&=  L_{D_{\cV_g}}([f^{\cE_f}] ) \otimes \theta \\
& = \frac{1}{\abs{\cU_{\cV_g}}}\sum_{\sigma \in \cU_{\cV_g}} \sigma L_{\cV_f}([f^{\cE_f}]) \otimes \theta \\
& = \frac{1}{\abs{\cU_{\cV_g}}}\sum_{\sigma \in \cU_{\cV_g}} \sigma L_{\cV_f}([f^{\cE_f}]  \otimes \theta)  &\text{by Lemma~\ref{lem:unipotent-general} } \\
& \in M_{n, \Theta}.
\end{align*} Since $\theta$ is arbitrary, we conclude that $M_d/M_n$ is supported in degrees $< \dim_{\aF}(\aF^n + \cE_f)$, and hence is torsion. Clearly, we have \[ \dim_{\aF}(\aF^n + \cE_f) = n + q^{d(n-d)}(1 + q + \cdots + q^{d-1}).  \]  This completes the proof.
\end{proof}


\begin{corollary} 
\label{cor:L-is-irreducible}	
Let $\Theta$ be an irreducible representation of $\GL_d$. Then, in $\ModVI^{\gen}$, $M_{n, \Theta}$ is irreducible  and is the socle of $\cI(\Theta)$.
\end{corollary}
\begin{proof} This follows immediately from Theorem~\ref{thm:cofinal} and Theorem~\ref{thm:collapse}.
\end{proof}



\section{Classification of irreducibles}
\label{sec:classification}

Throughout, in this section, we shall assume that $q$ is invertible in $\bk$. We will now make heavy use of the first paper \cite{VI1} in this sequel. We have kept all of its  notation. In particular, we need the following notation: \begin{itemize}
	\item $\Gamma(M)$ is the maximal torsion submodule of $M$.
	\item $\bS$ is the saturation functor and is  given by the composition $\rS \circ \rT$ of section and localization functors from the introduction.
	
	\item $\Sigma, \newshift$ are two endofunctors on $\Mod_{\VI}$ called shift functors. Both of them commute with $\Gamma$; see \cite[Proposition~4.27]{VI1}.
\end{itemize}

The following theorem summarizes some of the results from the first paper that we need:

\begin{theorem}[\cite{VI1}] 
\label{thm:vi-summary}	
Let $M$ be a finitely generated $\VI$-module. Then we have the following \begin{enumerate}[\rm \indent (a)]
		\item  There is a polynomial $P$ of such that for $n \gg 0$ we have
		$\dim_{\bk} M(\aF^n) = P(q^n).$  The invariant $\deg P$ is called the $\delta$-invariant and is denoted by $\delta(M)$. If $\Theta$ is a representation of $\GL_d$, then we have $\delta(\cI(\Theta)) = d$.
		
		\item For large enough $d$, $\Sigma^d M$ and $\newshift^d M$ are semi-induced. Moreover, $\delta(M) = \delta(\newshift^d M) = \delta(\Sigma^d M)$ for any $d$.
		
		\item  There is an exact triangle \[\rR\Gamma(M) \to M \to \rR \bS(M) \to \] where $\rR \Gamma (M)$ is represented by a finite complex of finitely generated torsion modules and $\rR \bS (M)$ is represented by a finite complex of finitely generated induced modules with $\delta$-invariant at most $\delta(M)$.
		
		\item If $T$ is torsion and $I$ is induced then $\rHom(T, I) = 0$.
		
		\item Let $\Theta$ be a representation of $\GL_d$. Suppose $I$ is an induced module with $\delta(I) < d$. Then  $\rHom(I, \cI(\Theta)) = 0$.
		
		\item We have  $\delta(\bS(M)) = \delta(M)$. 
	\end{enumerate}
\end{theorem}

Here is an immediate corollary.

\begin{corollary}
	Let $M$ be a finitely generated $\VI$-module with $\delta(M) < 0$, and let $\Theta$ be a representation of $\GL_d$. Then $\rHom(M, \cI(\Theta)) = 0$. 
\end{corollary}

\begin{proposition}
	Let $\Theta$ be an irreducible representation of $\GL_d$. Suppose $M \subset \cI(
	\Theta)$ is any nonzero submodule. Then $\delta(M) = d$. 
\end{proposition}
\begin{proof}
Suppose, if possible,  $\delta(M)<d$. The previous corollary implies that $\rHom(M, \cI(\Theta)) = 0$, a contradiction because $\Hom(M, \cI(\Theta)) \neq 0$.
\end{proof}

\begin{proposition}
	Let $\Theta$ be an irreducible representation of $\GL_d$. Suppose $M \subset \cI(
	\Theta)$ is any nonzero submodule. Then  $\delta(\cI(
	\Theta)/ M)<d$.
\end{proposition}
\begin{proof}
	By Theorem~\ref{thm:vi-summary}(b), there is a $d$ such that $\newshift^d M$ is semi-induced. By Theorem~\ref{thm:vi-summary}(b) and the previous proposition, we have $\delta(\newshift^d  M) = d$. Since $\newshift^d  M$ is a submodule of $\newshift^d  \cI(\Theta)$, we see that the top part of $\newshift^d  M$ must be $\cI(\Theta)$. Since $\delta((\newshift^d  \cI(\Theta))/\cI(\Theta)) <d$, we conclude that  $\delta(\newshift^d(\cI(
	\Theta)/ M)) < d$. The result follows from Theorem~\ref{thm:vi-summary}(b).
\end{proof}

By part (e) of Theorem~\ref{thm:vi-summary}, we see that the $\delta$-invariant on $\VI$-modules descends to the category of generic $\VI$-modules.  In other words, for a generic $\VI$-module $N$, we can define $\delta(N) = \delta(\rS(N))$. Then for any finitely generated $\VI$-module $M$, we have $\delta(\rT(M)) = \delta(\bS(M)) = \delta(M)$. 


\begin{theorem} \label{thm:irreducibles}
	Set $\rL(\Theta) = \rT(M_{d, \Theta})$. We have the following:
	\begin{enumerate}[\rm \indent (a)]
		\item Let $\Theta$ be an irreducible representation of $\GL_d$. Then $\rL(\Theta)$  is irreducible and is the socle of $\rT(\cI(\Theta))$. Moreover, $\rT(\cI(\Theta))$ is of finite length.

		\item  $\Mod_{\VI}^{\gen}$ is artinian. 
		
		\item The correspondence \[\bigsqcup_{n \ge 0} \Irr(\Mod_{\bk[\GL_n]}) \to \Irr(\Mod_{\VI}^{\gen})  \] given by $\Theta \mapsto \rL(\Theta)$  is one-to-one.
	\end{enumerate}
\end{theorem}
\begin{proof}
	Proof of (a).  By Corollary~\ref{cor:L-is-irreducible}, we see that $\rL(\Theta)$  is irreducible and is the socle of $\rT(\cI(\Theta))$. For the second assertion, we proceed by induction on $d$.     By the previous proposition, there is an $m \gg 0$ such that $\newshift^m (\cI(\Theta)/M_{d, \Theta})$ admits a filtration by induced modules of the form $\cI(Z)$ where $Z$ is a $\GL_k$ representation for  some $k < d$. By induction, $\rT(\cI(Z))$ is of finite length. The same must hold for $\rT(\newshift^m (\cI(\Theta)/M_{d, \Theta}))$. Since $\rT(\cI(\Theta))/\rL(\Theta)$ embeds into $\rT(\newshift^m (\cI(\Theta)/M_{d, \Theta}))$, we conclude that $\rT(\cI(\Theta))/\rL(\Theta)$ is of finite length. Thus (a) holds.
	
	Proof of (b). Let $M$ be a finitely generated $\VI$-module.  Since $\rT(M)$ embeds into $\rT(\newshift^m M)$ it suffices to show that $\rT(\newshift^m M)$ is of finite length for some $m$. Let $m \gg 0$ be such that $\newshift^m M$ is semi-induced. By part (a), $\rT(\newshift^m M)$ is of finite length for such an $m$. This shows that (b) holds.

	Proof of (c). Let $M$ be a finitely generated $\VI$-module such that $\rT(M)$ is irreducible. We may assume that $M$ is torsion free. Let $m \gg 0$ be such that $\newshift^m M$ is semi-induced. Let \[ 0 = M^0 \subset \ldots \subset M^r = \newshift^m M \] be a filtration such that $M^i/ M^{i-1} \cong \cI(\Theta_i)$ for some $\GL_{i}$ representation $\Theta_i$. Since $M$ is torsion-free it embeds into $\newshift^m M$. Pick the least $i$ such that $M \cap M^i$ is nonzero. By irreducibility of $\rT(M)$, the module $M/(M\cap M^i)$ is torsion. Also, $M\cap M^i$ embeds into $\cI(\Theta_i)$. Thus $\rT(M) = \rT(M \cap M^i)$ must be the socle $\rL(\Theta_i)$ of $\rT(\cI(\Theta_i))$. Thus the correspondence $\Theta \mapsto \rL(\Theta)$ is surjective. 
	
To see injectivity,  first note that if $\rL(\Theta_1)$ and $\rL(\Theta_2)$ are isomorphic then their $\delta$-invariants must be equal. Assume that it is $d$. Let $f \colon \rL(\Theta_1) \to \rL(\Theta_2)$ be an isomorphism.  By definition of the Serre quotient, this map is induced by a map $g \colon A \to B$ where $A \subset M_{d, \Theta_1}$ and $B \subset M_{d, \Theta_2}$ are submodules satisfying $B, M_{d, \Theta_1}/A \in \Mod_{\VI}^{\tors}$. The latter two modules are supported in finitely many degrees. Thus if $m$ is large enough, we see that $\newshift^m g \colon \newshift^m M_{d, \Theta_1} \to \newshift^m M_{d, \Theta_2}$ is an isomorphism. Further, we can assume that $m$ is large enough so that $\newshift^m M_{d, \Theta_1}$ and $\newshift^m M_{d, \Theta_2}$ are semi-induced. The top parts of these semi-induced modules are $\cI(\Theta_1)$ and $\cI(\Theta_2)$ respectively (the same argument as in the proof of the previous proposition holds). Since $\newshift^m g$ is an isomorphism, we see that it induces an isomorphism  $\cI(\Theta_1) \to \cI(\Theta_2)$. This implies that $\Theta_1$ and $\Theta_2$ are isomorphic. This completes the proof of (c).
\end{proof}

For a $\VI$-module $M$, let $\cK_d(M)$ denote the intersection of kernels of maps from $M$ to $\VI$-modules generated in degrees $<d$. In other words, we have \[ \cK_d(M)  = \bigcap_{\substack{f \colon M \to N  \\ 
t_0(N) < d}} \ker f. \] Similarly, we define $\cK^{\sat}_d(M)$ to be the  intersection of kernels of maps from $M$ to semi-induces $\VI$-modules generated in degrees $<d$.

\begin{lemma}
	The shift $\Sigma^m$ commutes with $\cK_d$ and $\cK^{\sat}_d$ in the category of generic $\VI$-modules, that is, we have the following: \begin{enumerate}[\rm \indent (a)]
		\item $\rT\Sigma^m\cK_d(M) = \rT \cK_d \Sigma^m(M)$. 
		\item $\rT\Sigma^m\cK^{\sat}_d(M) = \rT \cK^{\sat}_d \Sigma^m(M)$. 
	\end{enumerate} 
\end{lemma}
\begin{proof}
Proof of (a). We first show that $\cK_d \Sigma^m(M) \subset \Sigma^m \cK_d(M) \subset \Sigma^m M$. Suppose $x \in \cK_d \Sigma^m(M)$ and let $g \colon M \to N$ be a map with $t_0(M) < d$. Since $t_0(\Sigma^m N) < d$ and $x \in \cK_d \Sigma^m(M)$, we see that $x \in \Sigma^m \ker g.$ Thus we have \begin{align*}
x &\in \bigcap_{\substack{g \colon M \to N  \\ 
		t_0(N) < d}} \Sigma^m \ker g\\
	& =  \Sigma^m \bigcap_{\substack{g \colon M \to N  \\ 
			t_0(N) < d}}  \ker g \\
		& = \Sigma^m \cK_d(M),
\end{align*}  proving our claim.

Let $\iota \colon M \to \Sigma^m M$ be the map induced by the $\VI$-morphism $0 \to \aF^m$. Suppose $x \in \Sigma^m \cK_d(M)$. Let $N$ be a $\VI$-module generated in degrees $< d$,  and let  $g \colon \Sigma^m M \to N$ be arbitrary. Then  we have $\Sigma^m(g \circ \iota) (x) = 0$ (note that $\Sigma^m(g \circ \iota)\colon \Sigma^m M \to \Sigma^m N$). But $\Sigma^m(g \circ \iota) (x) = \Sigma^m(g) \circ \Sigma^m(\iota) (x)$. Since $g$ is arbitrary, we see that $\Sigma^m(\iota) (x) \in \Sigma^m \cK_d \Sigma^m M$. This implies $\Sigma^m (\frac{\Sigma^m \cK_d(M)}{\cK_d(\Sigma^m M)}) = 0$ and so $\frac{\Sigma^m \cK_d(M)}{\cK_d(\Sigma^m M)}$ is torsion completing the proof. 

Proof of (b). The same proof as in Part (a) works as semi-induced modules generated in degrees $< d$ are closed under shift. 
\end{proof}


\begin{theorem} Suppose $\Theta$ is an irreducible $\GL_d$ representation. Then we have \[\rT(\cK_d(\cI(\Theta))) = \rT(\cK^{\sat}_d(\cI(\Theta))) = \rL(\Theta)\] where $\rL(\Theta) = \rT(M_{d, \Theta})$, as in the previous theorem. Moreover, we have the following: \begin{enumerate}[\rm \indent (a)]
		\item $\rS(\rL(\Theta)) = \cK^{\sat}_d(\cI(\Theta))$. In other words, $\cK^{\sat}_d(\cI(\Theta))$ is the saturation of $M_{d, \Theta}$.
		\item $\deg \rR^i \Gamma(\cK^{\sat}_d(\cI(\Theta))) \le 2d - 2(i-1)$ for $i \ge 2$. In particular,  $\rR^i \Gamma(\cK^{\sat}_d(\cI(\Theta))= 0$ for $i > d + 1$.
		\item $\cK^{\sat}_d(\cI(\Theta))$ is generated in degrees $\le 2 d$. Moreover, $t_i(\cK^{\sat}_d(\cI(\Theta))) \le 2d -i$ for $i \ge 0$.
	\end{enumerate}
\end{theorem}
\begin{proof}
	It is clear that $\cK_d(\cI(\Theta)) \subset \cK^{\sat}_d(\cI(\Theta))$. We first show that $M_{d, \Theta} \subset \cK_d(\cI(\Theta))$. To see this suppose $\phi \colon \cI(\Theta) \to N$ be a map where $N$ is some $\VI$-module generated in degrees $< d$. Fix a $\VI$-morphism $f \colon \aF^d \to X$. Since $N$ is generated in degrees $< d$, there are $x_1, \ldots, x_r \in \sqcup_{k <d} N(\aF^k)$ and $g_1, \ldots, g_r \in \sqcup_{k < d} \Hom_{\VI}(\aF^k, X)$ such that \[ \phi([f] \otimes \theta) = \sum_{i=1}^r g_i(x_i). \] Then we have \begin{align*}
\phi(L_{\cV_f}([f^{\cE_f}] \otimes \theta)) & =  L_{\cV_f}(\phi([f^{\cE_f}] \otimes \theta))\\
& = L_{\cV_f}(\sum_{i=1}^r g_i^{\cE_f}(x_i)) 
	\end{align*}  This last expression vanishes because of Lemma~\ref{lem:distinct-images}  as $\im(f) \neq \im(g)$. Thus we have $M_{d, \Theta} \subset \cK_d(\cI(\Theta)) \subset \cK^{\sat}_d(\cI(\Theta))$.

	We now show that $\rT \cK^{\sat}_d(\cI(\Theta)) =  \rL(\Theta)$. It suffices to show that $\rT \Sigma^m (\cK^{\sat}_d(\cI(\Theta))/ M_{d, \Theta}) = 0$ for large enough $m$. By the previous lemma, we have \[\rT \Sigma^m (\cK^{\sat}_d(\cI(\Theta))/ M_{d, \Theta}) = \rT \Sigma^m \cK^{\sat}_d(\cI(\Theta)/ M_{d, \Theta}) =  \rT \cK^{\sat}_d( \Sigma^m (\cI(\Theta)/ M_{d, \Theta})).\] By Theorem~\ref{thm:irreducibles}(a), $\cI(\Theta)/ \rL(\Theta)$ has $\delta$-invariant $<d$. Thus by Theorem~\ref{thm:vi-summary}(b), we see that $\Sigma^m(\cI(\Theta)/ \rL(\Theta))$ is a semi-induced module generated in degrees $<d$ for any large enough $m$. This implies that $\cK^{\sat}_d( \Sigma^m (\cI(\Theta)/ \rL(\Theta))) = 0$, completing the proof.
	
	Proof of (a). Since $\rT(\cK^{\sat}_d(\cI(\Theta))) = \rL(\Theta)$, it suffices to show that $\cK^{\sat}_d(\cI(\Theta))$ is saturated. It is clearly torsion-free. Since $\cI(\Theta)$ is derived saturated and $\cK^{\sat}_d(\cI(\Theta))$ is a torsion-free submodule of it, it suffices to show that $\cI(\Theta)/\cK^{\sat}_d(\cI(\Theta))$ is torsion-free. But this is clear because by definition of $\cK^{\sat}_d$, $\cI(\Theta)/\cK^{\sat}_d(\cI(\Theta))$ is a submodule of a direct sum of semi-induced modules. This proves (a).
	
	Proof of (b). It follows immediately from \cite[Corollary~5.2]{VI1} and \cite[Theorem~1.1(1)]{GL}. A more detailed argument in the $\FI$-module case is provided in  \cite[Theorem~2.10(4)]{linearranges}. 
	
	Proof of (c). By parts (a) and (b), the quantity $r(\cK^{\sat}_d(\cI(\Theta)))$ as in  \cite[Theorem~5.13]{VI1} is bounded by $2d$. By Theorem~\ref{thm:vi-summary}(c), we see that $t_0(\rR \bS(\cK^{\sat}_d(\cI(\Theta)))) \le d$. Thus the proof of \cite[Theorem~5.13]{VI1} gives us $t_i(\cK^{\sat}_d(\cI(\Theta))) \le 2d -i$ for $i \ge 0$. This completes the proof. 
\end{proof}

Theorem~\ref{intro:thm:irreducible-correspondence}  and Theorem~\ref{intro:thm:saturation} now follow by the previous two theorems after noting that $\cL(\Theta) = \cK^{\sat}_d(\cI(\Theta))$.

\end{document}